\documentclass[11pt]{amsart}

\usepackage{amsmath,amsfonts,amssymb,amsthm}
\usepackage{enumerate}
\usepackage{algorithm}
\usepackage[noend]{algpseudocode}
\usepackage{tikz-cd}
\usepackage[hyperindex,breaklinks]{hyperref}
\hypersetup{
  colorlinks   = true, 
  urlcolor     = blue, 
  linkcolor    = blue, 
  citecolor    = black 
}

\usepackage[a4paper]{geometry}
\newtheorem{theorem}{Theorem}
\newtheorem{lemma}{Lemma}
\newtheorem{conjecture}[lemma]{Conjecture}

\newtheorem{corollary}{Corollary}

\newtheorem*{theorem*}{Theorem}
\theoremstyle{remark}

\newtheorem{question}[lemma]{Question}
\newtheorem*{observation*}{Observation}
\makeatletter
\renewcommand*{\@seccntformat}[1]{%
  \csname the#1\endcsname.
}
\renewcommand\section{\@startsection {section}{1}{\z@}
                                   {-3.5ex \@plus -1ex \@minus -.2ex}%
                                   {2.3ex \@plus.2ex}%
                                   {\centering\normalsize\scshape}}
\renewcommand{\subsection}[1]{%
  \pagebreak[2]
  \refstepcounter{subsection} \addcontentsline{toc}{subsection}{
    {\protect\makebox[0.3in][r]{\thesubsection.} #1}}
  \noindent
  \textbf{\thesubsection.\ {#1}.}
  \everypar={}%
}
\renewcommand{\subsubsection}[1]{%
  \pagebreak[2]
  \refstepcounter{subsubsection} \addcontentsline{toc}{subsubsection}{
    {\protect\makebox[0.3in][r]{\thesubsubsection.} #1}}
  \noindent
  \textbf{\thesubsubsection.\ {#1}.}
  \everypar={}%
}

\makeatother
\usepackage{xspace}

\begin{document}
\title{Infinitesimal part of Weak Lefschetz using Milnor $K-$ theory}  
\author{Jagannathan Arjun Sathyamoorthy }
\date{}
\maketitle
\begin{abstract}
Let $X$ be a smooth projective variety over a field of characteristic $0$ and $L$ an ample divisor. In this paper we study the Weak Lefschetz conjecture for Chow groups using the technique employed by Grothendieck in his study of the problem for Picard groups, and using Bloch's formula to interpret Chow groups in terms of Milnor $K$- theory we prove the infinitesimal part of the conjecture.  
\end{abstract}
We begin by stating the well known conjecture 
\begin{conjecture}[Weak Lefschetz] For $X/k$ a smooth projective variety where $char(k)=0$ and $L$ is an ample divisor, the natural restriction map induces 
\begin{equation}
{
CH^p(X) \to CH^p(L)
}
\end{equation}   
an isomorphism when $2p<dim(L)$ and is injective at $2p= dim(L)$ modulo torsion.

\end{conjecture}

For the case $p=1$ it is the Grothendieck-Lefschetz theorem, which is proved as follows:

The Chow group of divisors can be identified with the Picard group. Grothendieck shows that each of the arrows in the following sequence are isomorphisms $Pic(X) \to Pic(\widehat{X}) \to Pic(L)$ where $\widehat{X}$ is the formal completion of $X$ along $L$. He achieves this by lifting $L$ to the formal completion $\widehat{X}$ and then using Lefschetz conditions to extend to whole of $X$.  

We denote the Milnor $K$-groups by $K_{*,M}$ and always consider $K_{*,M}$ only up to torsion, i.e., we denote $K_{p,M}(A) = K_{p,M} (A)\otimes_{\mathbb{Z}} \mathbb{Q}$. So we  ask the following question:
\begin{question}
Let us consider $X/k$ an algebraic variety s.t. $\dim(X) \geq 4$ and $char(k)=0$, and consider an ample divisor $L$. We denote by $\mathcal{K}_{p,M}$ the Milnor $K$-sheaf which is the sheaf associated to the presheaf $ U \mapsto K_{p,M}(\mathcal{O}_X(U))$ where $K_{p,M}(A)$ is the Milnor $K$ theory of the ring $A$. We then want to know when
\begin{equation}
{
H^q(\widehat{X}, \varprojlim \mathcal{K}_{p,M}(\mathcal{O}_{L_n})) \to H^q(L, \mathcal{K}_{p,M}(\mathcal{O}_L))
}
\end{equation}
is an isomorphism?
\end{question}
The idea would be to try to use the same technique on the above question as Grothendieck's algebraization theorem.  We study the cohomology of the sheaf $\Phi_{p,n}$ defined as 
\begin{equation}
{
0 \to \Phi_{p,n} \to \mathcal{K}_{p,M}({\mathcal{O}_{L_{n+1}}}) \to \mathcal{K}_{p,M}({\mathcal{O}_{L_{n}}}) \to 0
}
\end{equation}
and explicitly calculate it and see when it is zero. 

We can study just at the level at stalks and hence can localize to a suitable neighbourhood of a point $*$ of $L$. This gives us the following sequence
\begin{equation}
{
0 \to \phi_{p,n} \to K_{p,M}(A/\sigma^{n+1}) \to K_{p,M}(A/\sigma^n) \to 0
}
\end{equation} 
where $\sigma$ is the local generator at the point $*$ of the locally principal ideal $I$ and $A$ is the localization at the prime corresponding to $*$ of the $k-$algebra $k[x_1,...,x_{\dim(X)}]/(\sigma)$ i.e. the local ring of $L$ at $*$. 

We predominantly use ideas from Bloch, S. \cite{B}
We start by understanding the case for a truncated polynomial algebra $A[\sigma]/\sigma^n$ where $\sigma$ is an indeterminate and by abuse of notation continue using $\phi_{p,n}$.  This will allow us to use van der Kallen's result stated in Bloch's paper
\begin{theorem}[van der Kallen]

Let $A$ be a local ring with $1/2 \in A$ and $\sigma$ an indeterminate. Then the Formal Tangent Space $TK_{p,M}(A)$ defined as 
\begin{center}
{
$TK_{p,M}(A) = ker \; \{ K_{p,M} (A[\sigma])/(\sigma^2) \to K_{p,M}(A) \}$
}
\end{center}
has a natural $A$-module structure and is $A$-isomorphic to $\Omega^{p-1}_{A/\mathbb{Q}}$.

\end{theorem}
We omit the proof in the following exposition as one can find it Green and Griffiths \cite{GG}. We now state a corollary to van der Kallen's result also stated by Bloch \cite{B}
\begin{corollary}
{
$\phi_{2,n} = ker \; \{ K_{2,M}(A[\sigma]/\sigma^{n+1}) \to K_{2,M}(A[\sigma]/\sigma^n) \}$ is generated by symbols of the form $\{ 1 + c \sigma^n, a \}$ and $\{ 1 + e\sigma^n, 1 - \sigma \} $ for $a,e \in A^*$ and $c \in A$
}
\end{corollary}
\begin{proof}
 As $TK_{2,M}(A[\sigma]/\sigma^{n+1}) = ker \{K_{2,M}((A[\sigma]/\sigma^{n+1})[\varepsilon]/\varepsilon^2) \to K_{2,M}(A[\sigma]/\sigma^{n+1}) \}$  \\
we consider $\pi: TK_{2,M}(A[\sigma]/\sigma^{n+1}) \to \phi_{2,n}$ which is obtained by sending $\varepsilon \mapsto \sigma^n$.  We know that
\begin{equation}
{
\Omega^1_{A[\sigma]/\sigma^{n+1}} \cong (\Omega^1_A \otimes \mathbb{Q}[\sigma]/\sigma^{n+1}) \oplus (A \otimes \Omega^1_{\mathbb{Q}[\sigma]/\sigma^{n+1}})
}
\end{equation}
 and van der Kallen tells us that $TK_{2,M}(A[\sigma]/\sigma^{n+1})$ is generated by symbols $\{1+c\varepsilon, a\}$ and $\{ 1 + e \varepsilon , 1 - \sigma \}$ for $c,e \in A[\sigma]/\sigma^{n+1}$ and $a \in A^*$. The map $\varepsilon \mapsto \sigma^n$ sends $\sigma \varepsilon \mapsto 0$ implying surjectivity of $\pi$ finishing the corollary. 
 \end{proof}

We can extend the result to all $ p \geq 3$  i.e. $\phi_{p,n} = ker \{ K_{p,M}(A[\sigma]/\sigma^{n+1}) \to K_{p,M}(A[\sigma]/\sigma^n)$ is generated by symbols of the form $\{ 1 + c \sigma^n, u_1, ..., u_{p-1}\}$ and $\{1 + e \sigma^n, 1 - \sigma, u_2, ..., u_{p-1}\}$ where $u_i, e \in A^*$ and $c \in A$ as 
 \begin{equation}
{
\Omega^p_{A[\sigma]/\sigma^{n}} \cong (\Omega^p_A \otimes \mathbb{Q}[\sigma]/\sigma^{n}) \oplus (\Omega^p_A \otimes \Omega^1_{\mathbb{Q}[\sigma]/\sigma^{n}})
}
\end{equation}

Because $A$ is local, $e$ or $e+1$ is invertible, we have the following identity $\{1+ (1+e)\sigma, 1 - \sigma \} = \{ 1 + e \sigma , 1 - \sigma \} + \{ 1+ \sigma^n, 1 - \sigma \}$.  
 
Given that we now know the structure of $\phi_{p,n}$ we can formulate the following theorem 
\begin{theorem}

The map $\phi_{p,n} \to \Omega^{p-1}_A \otimes (\sigma^n / \sigma^{n+1})$ given by $\{ 1 + c \sigma^n, u_1,..., u_{p-1} \} \mapsto du_1/u_1 \wedge ... \wedge du_{p-1}/u_{p-1} \otimes c\sigma^n$ is an isomorphism of $A$ modules. 

\end{theorem}
\begin{proof} The idea works upon Bloch's ideas \cite{B}, and what we would have to do is show that symbols of the form $\{ 1 + e\sigma^n, 1 - \sigma\}$ where $e \in  A$ vanish for $p=2$ case and the higher cases follow analogously.

We begin by looking at the group $K_{2,M}(A[[\sigma]][1/\sigma])$ where $\{ 1+c\sigma^{n+1}, c \sigma^{n+1}\} =0$ which gives us $(n+1)\{1 + c \sigma^{n+1}, \sigma \} =  - \{ 1+ c\sigma^{n+1}, c\}$. 
We also have \par
$\{ 1 + c \sigma^{n+1}, \sigma \} = \{1 - \sigma (1 + c\sigma^{n+1} - c \sigma^n), \sigma \} = \{ (1-\sigma)(1+c\sigma^{n+1}), 1+c\sigma^{n+1} - c\sigma^n \}$. This gives us that 
\begin{equation} \label{eq}
{
\{ 1 + c\sigma^{n+1}, c\} = (n+1)\{ (1-\sigma)(1+c\sigma^{n+1}), 1 + c\sigma^{n+1} - c \sigma^n \}
}
\end{equation}
Then all one would have to do is show that Eq.~\ref{eq} survives in $K_{2,M}(A[[\sigma]])$. For this we can use M. Kerz's result in his thesis \cite{MK}
\begin{theorem}[Kerz]
 Let $A$ be a regular connected semi-local ring containing a field with quotient field $F$. Assume that each residue field of $A$ has more than $M_n$ elements. Then the map 
$i_n : K_{n,M}(A) \to K_{n,M}(F)$ is universally injective. 
\end{theorem}

What the above theorem tells us is that if we have $R$-local ring and $\kappa_R$ the residue field of $R$ s.t $\text{char}(\kappa_R)=0$ and $g \in R$ the canonical map $K_{p,M}(R) \to K_{p,M}(R_g)$ is injective. In our case we have that $R=A[[\sigma]]$ and $g=\sigma$ and hence equation~\ref{eq} survives in $K_{2,M}(A[[\sigma]])$. 

Now we can project using $Eq.7$ under the surjective map induced by Milnor-K theory by the canonical projection of $A[[\sigma]] \to A[\sigma]/\sigma^{n+1}$ and we get the following in $K_{2,M}(A[\sigma]/\sigma^{n+1})$ i.e 
\begin{equation}
{
0 = \{1,c\} = (n+1) \{1-\sigma, 1 - c \sigma^n \} = \{ 1-\sigma, (1-c\sigma^n)^{n+1}\} = \{1-\sigma, 1-(n+1)c\sigma^n \}
}
\end{equation}
which is what we wanted and hence solves the entire theorem i.e. telling us that  $\phi_{p,n}$ is generated by symbols of the form $\{1+c\sigma^n, a\}$. 
\end{proof} 

In our case $\sigma$ is an element of $A$.  The previous construction also leads to the correct result because for any transcendental $\lambda$ we have the map $\tau : A[\lambda]/\lambda^n \to A/\sigma^n$ under the map $\tau(\lambda) = \sigma$ is surjective and also remains surjective in Milnor K-theory. Hence the isomorphism for $\sigma \in A$ will be as follows: 
\begin{equation}
{
\phi_{p,n} = ker \{ K_{p,M}(A/\sigma^{n+1}) \to K_{p,M}(A/\sigma^n) \} \cong \Omega^{p-1}_{(A/\sigma)/\mathbb{Q}} \otimes_{A/\sigma} (\sigma^n/\sigma^{n+1})
}
\end{equation}
This gives us what we sought out to find. Now we head back to the main problem, i.e.
\begin{equation}
{
H^q(\widehat{X}, \varprojlim \mathcal{K}_{p,M}(\mathcal{O}_{L_n})) \to H^q(L, \mathcal{K}_{p,M}(\mathcal{O}_L))
}
\end{equation}
is an isomorphism.
We start by taking the exact sequence 
\begin{equation}
{
0 \to \bigoplus_{r+s = p-1} \Omega^r_{k/\mathbb{Q}} \otimes \Omega^s_{L/k} \to \Omega^{p-1}_{\mathcal{O}_L/\mathbb{Q}} \to \Omega^{p-1}_{\mathcal{O}_L/k} \to 0
}
\end{equation}
As $L$ is ample we have $\mathcal{I}^n/\mathcal{I}^{n+1} = \mathcal{O}_L(-nL)$ where we denote by $\mathcal{I}$ the ideal sheaf of $L$. So tensoring the above sequence with $\mathcal{O}_L(-nL)$ and observing that $\mathcal{O}_L(-nL)$ is dual to the ample divisor $L$ we can use the Kodaira Vanishing Theorem along with Serre Duality which gives us 
\begin{equation}
{
H^q(L, \Omega^r_{L/k} \otimes \mathcal{O}_L(-nL))=0
}
\end{equation}
when $q+r \leq \dim(L) -1$. Now as $\Omega^r_{k/\mathbb{Q}}$ is a constant sheaf we have that $H^q(L, \Omega^r_{k/\mathbb{Q}}\otimes \mathcal{W}) = \Omega^r_{k/\mathbb{Q}} \otimes H^q(L,\mathcal{W})$ for any sheaf $\mathcal{W}$. Taking the long cohomology sequence we get 
\begin{equation}
{
H^q(L, \Omega^{p-1}_{\mathcal{O}_L/\mathbb{Q}} \otimes \mathcal{O}_L(-nL))=0
}
\end{equation}
when $p+q \leq \dim(L)$. 
\par
From the long cohomology induced by 
\begin{equation}
{
0 \to \Phi_{p,n} \to \mathcal{K}_{p,M}(\mathcal{O}_{L_{n+1}}) \to \mathcal{K}_{p,M}(\mathcal{O}_{L_n}) \to 0
}
\end{equation}
and from the $\mathcal{O_L}$ module isomorphism in Theorem 2 and from the above Eq 14 we have that 
\begin{equation}
{
H^q(L, \mathcal{K}_{p,M}(\mathcal{O}_{L_{n+1}})) \to H^q(L, \mathcal{K}_{p,M}(\mathcal{O}_{L_n}))
}
\end{equation} 
is an isomorphism for $p=q \leq dim(L) -1 $ and is injective for $p+q < \dim(L)$. Now we observe the natural map
\begin{equation}
{
H^q(\widehat{X}, \varprojlim \mathcal{K}_{p,M}(\mathcal{O}_{L_n})) \to \varprojlim H^q(L, \mathcal{K}_{p,M}(\mathcal{O}_{L_n}))
}
\end{equation}
satisfies the conditions given in EGA-III \cite[Chapter 0, Theorem 13.3]{G}. We state if for convenience
\begin{theorem}[Mittag-Leffler and Cohomology]
{
Let $X$ be a topological space, with $(\mathcal{F}_{k})_{k \in \mathbb{N}}$ a projective system of sheaves of abelian groups on $X$ and let $\mathcal{F} = \varprojlim \mathcal{F}_k$ Assume the following conditions are satisfied:\\
(1) There is a base $\mathcal{B}$ of the topology of $X$ such that for all $U \in \mathcal{B}$ and $i \geq 0$ the projective system $(H^i(U,\mathcal{F}_k))_{k \in \mathbb{N}}$ satisfies Mittag-Leffler. \\ 
(2) For every $x \in X$ and every $i > 0$, $\varinjlim_U ( \varprojlim_k H^i(U))=0$ where $U$ runs over all neighbourhoods of $x$ belonging to $\mathcal{B}$. \\
(3) The homomorphisms $u_{h,k} : \mathcal{F}_k \to \mathcal{F}_h \; (h \leq k)$  defining the projective system $(\mathcal{F}_k)$ are surjective. 
\\
Under these conditions, for all $i>0$, the canonical homomorphism 
\begin{equation}
{
h_i : H^i(X,\mathcal{F}) \to \varprojlim H^i(X,\mathcal{F}_k)
}
\end{equation}
is surjective. If in addition for a value of $i$ the projective system $(H^{i-1}(X,\mathcal{F}_k))_{k \in \mathbb{N}}$ satisfies Mittag Leffler, then $h_i$ is bijective.
}
\end{theorem}
Both $(\mathcal{K}_{p,M}(\mathcal{O}_{L_n}))_n$ and $H^r(L, \mathcal{K}_{p,M}(\mathcal{O}_{L_n}))$ for $r+p \leq dim(L)-1$ are injective and hence we can take the basis of the topology consisting of affine open sets trivializing the locally free sheaf 
\begin{equation}
{
\Omega^{p-1}_{\mathcal{O}_L/k} \otimes \mathcal{O}_L(-nL) 
}
\end{equation}
which for all $n$ clearly exists. For $U$ such an open set 
\begin{equation}
{
H^i(U, \Omega^{p-1}_{\mathcal{O}_L/k} \otimes \mathcal{O}_L(-nL) |_U) = 0
}
\end{equation}
when $i \geq 1$ and $n \geq 1$. This gives us that $H^q(U, \mathcal{K}_{p,M}(\mathcal{O}_{L_n}))_n$ is a system of isomorphic groups.
The above argument gives us that 
\begin{equation}
{
\varprojlim H^q(L, \mathcal{K}_{p,M}(\mathcal{O}_{L_n})) \cong H^q(\widehat{X}, \varprojlim \mathcal{K}_{p,M}(\mathcal{O}_{L_n}))
}
\end{equation}
from Eq. 15, we have $H^q(\widehat{X}, \varprojlim \mathcal{K}_{p,M}(\mathcal{O}_{L_n})) \cong H^q(L,\mathcal{K}_{p,M}(\mathcal{O}_{L_n}))$, for $p+q \leq dim(L)-1$. Taking inverse limits we have 
\begin{equation}
{
H^q(\widehat{X}, \varprojlim \mathcal{K}_{p,M}(\mathcal{O}_{L_n})) \cong H^q(L, \mathcal{K}_{p,M}(\mathcal{O}_L))
}
\end{equation}
Similarly when $p+q = dim(L)$ we have injection 
\begin{equation}
{
H^q(L, \mathcal{K}_{p,M}(\mathcal{O}_{L_{n+1}})) \hookrightarrow H^q(L, \mathcal{K}_{p,M}(\mathcal{O}_{L_n}))
}
\end{equation} 
for all $n$. Consequently we have 
\begin{equation}
{
\varprojlim H^q(L, \mathcal{K}_{p,M}(\mathcal{O}_{L_n})) \hookrightarrow H^q(L, \mathcal{K}_{p,M}(\mathcal{O}_L))
}
\end{equation}
giving the injection and answering the question.
\\
Hence we can say that 
\begin{equation}
{
H^q(\widehat{X}, \varprojlim \mathcal{K}_{p,M}(\mathcal{O}_{L_n})) \cong H^q(L, \mathcal{K}_{p,M}(\mathcal{O}_L))
}
\end{equation}
is an isomorphism for all $p+q < dim(L)-1$ and is injective when $p+q = dim(L)$.
\par
Hence we have shown
\begin{theorem}
{
Let us consider $X/k$ a smooth algebraic variety of $dim \; X \geq 4$ and an ample divisor $L$ where $char(k)=0$. If we denote by $\mathcal{K_*}$ the $K$-theoretic sheaf and $\widehat{X}$ the formal completion of $X$ along $L$ then 
\begin{equation}
{
H^q(\widehat{X}, \varprojlim \mathcal{K}_{p}(\mathcal{O}_{L_n})) \to H^q(L, \mathcal{K}_{p}(\mathcal{O}_L))
}
\end{equation}
is an isomorphism for all $p+q < dim(L)-1$ and is injective when $p+q = dim(L)$.
}
\end{theorem}

\textbf{Acknowledgements:} I would like to thank Prof. Fedor Bogomolov from NYU and Prof. Johan de Jong from Columbia University for teaching me Algebraic Geometry and Algebraic $K$-Theory. I would like to thank Federico Buonerba for bringing this problem to my attention and am aware that he has worked on it and possibly has a different proof. I would also like to mention Prof. Vasudevan Srinivas and Prof. Amalendu Krishna for helping to formulate the problem better. Prof. Madhav Nori of the University of Chicago brought to my attention that Patel and Ravindra \cite{PR} had been working on the same problem. But our proofs differ. 
\\
\\

\textbf{Jagannathan Arjun Sathyamoorthy}. Contact email : arjun.jas@gmail.com

\end{document}